\documentclass[12pt, reqno]{amsart}
\usepackage{amsmath, amsthm, amscd, amsfonts, amssymb, graphicx, color}
\usepackage{setspace}
\usepackage{mathrsfs}
\usepackage{multicol}
\usepackage[bookmarksnumbered, colorlinks, plainpages]{hyperref}
\hypersetup{colorlinks=true,linkcolor=red, anchorcolor=green, citecolor=cyan, urlcolor=red, filecolor=magenta, pdftoolbar=true}

\newtheorem{theorem}{Theorem}[section]

\newtheorem{corollary}[theorem]{Corollary}
\theoremstyle{definition}

\theoremstyle{remark}

\numberwithin{equation}{section}

\newcommand{\NN}{\mathbb{N}}

\newcommand{\CC}{\mathbb {C}}

\newcommand{\D}{\mathbb{D}}
\begin{document}
\title[Supercyclicity and  resolvent condition ]{ Supercyclicity and  resolvent condition for   weighted composition operators}
\author[Tesfa  Mengestie*]{Tesfa  Mengestie*} \footnote{*Corresponding author}
\address{Department of Mathematical Sciences \\
Western Norway University of Applied Sciences\\
Klingenbergvegen 8, N-5414 Stord, Norway}
\email{Tesfa.Mengestie@hvl.no}
\author [Werkaferahu Seyoum ]{Werkaferahu Seyoum}
\address{Department of Mathematics,
Addis Ababa University, Ethiopia}
\email{Werkaferahu@gmail.com}

 \begin{abstract}
 For pairs of holomorphic maps $(u,\psi)$ on the complex plane, we study some dynamical properties of  the  weighted composition operator $W_{(u,\psi)}$  on the   Fock spaces.  We  prove   that  no weighted composition operator on  the   Fock spaces  is supercyclic. Conditions under which the operators satisfy the Ritt's resolvent growth  condition are also identified. In particular, we show  that a non-trivial composition operator  on the Fock spaces satisfies such a  growth condition if and only if it is compact.
\end{abstract}
\maketitle
\section{Introduction}\label{sec1}
For pairs of holomorphic maps $(u,\psi)$, several properties of   the  weighted composition operator $W_{(u,\psi)}: f\to u\cdot f\circ \psi$  in  various settings  are well understood. But there are still some   basic  structures of the operators   that require  further  investigations. In this paper, we study some dynamical structures of the operators on the   Fock spaces
 $\mathcal{F}_p, 0<p<\infty$. We recall that the space   $\mathcal{F}_p$  is the   space  consisting of all entire functions $f$ for which
 \begin{align*}
 \|f\|_{ p}^p= \frac{p}{2\pi} \int_{\CC} |f(z)|^p
e^{-\frac{p|z|^2}{2}}  dA(z) <\infty
\end{align*} where $dA$ is the usual  Lebesgue area measure on the complex plane $\CC$.

The space $\mathcal{F}_2$ is a reproducing kernel Hilbert space with kernel function $K_w(z)= e^{\overline{w}z}$ and normalized kernel  $k_w=\|K_w\|_2^{-1}K_w $. A straightforward calculation shows that for each $w\in \CC$, the function   $k_w$ belongs to   $\mathcal{F}_p$ with    $\|k_w\|_p= 1 $ for all $p$.

For each  entire function $f$,  the subharmonicity of   $|f|^p$ implies   the point estimate
 \begin{align}
\label{global}
 |f(z)|
 \leq e^{\frac{|z|^2}{2}}\bigg(\int_{D(z,1 )} |f(w)|^p
e^{-\frac{p|w|^2}{2}} dA(w)\bigg)^{\frac{1}{2}}\leq e^{\frac{|z|^2}{2}}  \|f\|_p
 \end{align} where  $D(z,1 )$ is a disc of radius $1$ and center $z$.

 \noindent
  The bounded and compact weighted composition operators on Fock spaces  were characterized  in terms of Berezin-type integral transforms in  \cite{TM4,UK} and the reference therein. Later,  Le \cite{Tle} considered the Hilbert space $\mathcal{F}_2$ and obtained a simpler condition namely  that  $W_{(u,\psi)}$ is bounded on $\mathcal{F}_2$ if and only if    $u\in \mathcal{F}_2$ and
\begin{align}
 \label{bounded}
 M(u, \psi):=\sup_{z\in \CC} |u(z)|e^{\frac{1}{2}(|\psi(z)|^2-|z|^2)} <\infty.
 \end{align} He further proved that
 \eqref{bounded} implies $
 \psi(z)= az+b$ for some complex numbers $a$ and $b$ such that $ |a|\leq 1$.   In \cite{TMMW}, T. Mengestie and M. Worku
 proved that the Berezin-type integral condition used to describe the boundedness of generalized Volterra-type integral operators $V_{(g,\psi)}$ on the Fock spaces  $\mathcal{F}_p$     is equivalent to  a simple condition similar to   \eqref{bounded}. Because of the Littlewood-Paley type description of the Fock  spaces, by  simply  replacing $|g'(z)|/(1+|z|)$ by $|u(z)|$ in the results there, it has been known   that \eqref{bounded}  in fact describes
 the bounded weighted composition operators on  all the spaces  $\mathcal{F}_p, \ 1\leq p< \infty, $  with norm bounds
\begin{align}
 \label{norm}
  M(u, \psi)\leq \|W_{(u,\psi)}\|\leq
  |a|^{-\frac{2}{p}}  M(u, \psi).\end{align}

    As indicated in \cite{Tle}, an  interesting consequence of \eqref{bounded} is that  if $|a|=1$, then a simple argument with Liouville's theorem leads to the explicit expression for the weight function
 \begin{align}
 \label{weight}
  u(z)= u(0)K_{-\overline{a}b}(z).
 \end{align}
 We note that $W_{(u,\psi)}$ can be written as the product $ M_u C_\psi$ where $ M_u $ and $C_\psi$ are respectively the multiplication and composition operators. By condition \eqref{bounded},  the operator $W_{(u,\psi)}$ can  be bounded   even if both the factors  $C_\psi$ and $M_u$ are  unbounded.  For example one can set  $u_0(z)= e^{-z}$, $\psi_0(z)= z+1$, and observe that   $W_{(u_0,\psi_0)}$ is bounded while both the factors  remain unbounded.  Similarly, compactness of $W_{(u,\psi)}$ has been described by the fact that $\psi(z)= az+b, |a|\leq 1$ and  $|u(z)|e^{\frac{1}{2}(|\psi(z)|^2-|z|^2)}\to 0$ as $|z|\to 0$.  Compactness implies that $|a|<1$.

 We conclude this section with a word on notation.  The notation
 $U(z)\lesssim V(z)$ (or
equivalently $V(z)\gtrsim U(z)$) means that there is a constant
$C$ such that $U(z)\leq CV(z)$ holds for all $z$ in the set of a
question. We write $U(z)\simeq V(z)$ if both $U(z)\lesssim V(z)$
and $V(z)\lesssim U(z)$.
\section{Supercyclic weighted composition operators on $\mathcal{F}_p$ }
The power bounded and uniformly mean ergodic dynamical properties of the weighted composition operators on $\mathcal{F}_p$ have been recently described in \cite{TW4}. One of the main objectives of this note is  to  take that  investigation further and study  the supercyclic structure and resolvent growth conditions of the operators on the spaces. We start  by  recalling certain definitions.
  A bounded linear operator $T$ on a   Banach space $\mathcal{X}$ is said to be  hypercyclic if there exists a vector $f$ in $\mathcal{X}$  for which  the orbit,  $Orb(T,f)= \{ T^n f:n\in \NN_0\}$,   is dense in  $\mathcal{X}$ where the operator $T^n$ is the n-th iterate of  $T$,  $T^0= I$    and    $I$ is the identity map on $\mathcal{X}$.   The operator is    supercyclic  if there exists a vector $f$ for which its    projective
orbit, the set of scalar multiplies of  $Orb(T,f)$,  is dense in $\mathcal{X}$. Clearly,   every hypercyclic operator is supercyclic but not conversely.

We note in passing that the dynamical properties of operators have been the  object of many investigations over  the past several  years  in part  because   these properties   intimately connected with the study of the famous  invariant subspace  problem . %We refer to the surveys  \cite{BM, Grosse} for the basics.

 The study of  the dynamical properties of an operator is related to the study of  its  iterates. For  the operator $W_{(u,\psi)}$ and an entire function $f$, a simple computation shows that each element of the orbit has the form
\begin{small}
\begin{align}
\label{interplay}
W_{(u,\psi)}^n f=  f(\psi^n) u_n, \ \ u_n:= \prod_{j=0}^{n-1} u(\psi^j)
\end{align}
\end{small} for each   $n\in \NN$ and $\psi^0= I$  the identity map on $\CC$. In \cite{TW2,TW1},  it was shown that the Fock spaces
support no supercyclic and hence hypercyclic composition operators.  Motivated by this, one may ask  whether  the interplay between $u$ and $\psi$  in \eqref{interplay} results in supercyclic weighted composition operators on   Fock spaces as it does for example in the case of boundedness and compactness operator theoretic structures.
  Our next main  result affirms that the spaces support no supercyclic weighted composition operator either.
\begin{theorem}\label{supercyclic}
 Let $1\leq p<\infty$ and $(u,\psi)$ be a pair  of entire functions on $\CC$ which induces a bounded  weighted composition operator $W_{(u,\psi)}$ on $\mathcal{F}_p$. Then  $W_{(u,\psi)}$ can not be  supercyclic on $\mathcal{F}_p$.
\end{theorem}

  %As will be seen in the proof, the existence of nonzero eigenvalues of  $W_{(u,\psi)}$
%makes it possible to  inherits the non-supercyclic properties of the composition operator.

As pointed above if $u= 1$, then $W_{(u,\psi)}$ is just the composition operator $C_\psi$. On the other hand, if $\psi$ is the identity map, then
                $W_{(u,\psi)}$ reduces to the multiplication operator $M_u$. With this, we obtain the following immediate consequence of Theorem~\ref{supercyclic}.
\begin{corollary}
Let $1\leq p<\infty$ and $(u,\psi)$ be a pair of entire functions on $\CC$, and let the operators  $C_\psi$ and $M_u$ be bounded on $\mathcal{F}_p$. Then, both  $C_\psi$ and $M_u$ are not supercyclic on $\mathcal{F}_p$.
 \end{corollary}
 The  results above  assert that the projective orbit of any given vector under $W_{(u,\psi)}$ is not large  enough   to  contain the space $\mathcal{F}_p$. The weighted operators exhibit the  same  supercyclic   phenomena as the unweighted  composition operator $C_\psi$. Since the weight function had  no effect on this property here, it is  interesting to ask what type of general Fock spaces   do support suppercyclic  composition operators. When the weight functions generating the spaces grow faster than  the classical weight function $|z|^2/2 $, it was verified in \cite{TMMW} that the corresponding Fock-type spaces fail to support. If the weight functions grow slower than the classical case and  the derivative of the weight function grows to infinity, the approaches used in  \cite{TMMW} show that the same conclusion follows. On the other hand,  when  the weight function  is $\psi_m(z)= |z|^m, 0<m\leq 1$, then as shown in \cite{KG},  each nontrivial  translation operator acting on the corresponding Fock-type spaces is hypercyclic and hence supercyclic. Since a bounded $C_\psi$ on such spaces  happens if and only if $\psi(z)= az+b $ with $|a|\leq 1$, by setting $a=1$ and $b\neq 0$, we observe that $C_\psi$ reduces to the translation operator $T_b$ which is hypercyclic. This in fact gives spaces where the composition operator admits hypercyclic structure and hence is supercyclic. Whether this extends to the non-trivial  weighted composition operators remains to be studied. \\
 \ \ \\
 \noindent \textbf{Proof of Theorem~\ref{supercyclic}.}
  We now  give the proof of the theorem. Since $W_{(u,\psi)}$ is bounded, we set $\psi(z)= az+b$, with $|a|\leq 1$. For the case  $|a|<1$ or $|a|=1$ and $a\neq1$, we argue as follows. The map $\psi$ fixes the point $z_0= \frac{b}{1-a}$.   Assume on the contrary that there exists a supercyclic vector
$f$ in $\mathcal{F}_p$.  First we claim that $u$  is  zero free on $\CC$ because if u vanishes at the  point $w$, then  \eqref{interplay} implies that every element in the orbit of $f$ vanishes at $w$ which extends to the projective orbit which obviously  is not the case.  Observe also  that   $f$ can not have a zero in $\CC$. This is because  all the elements in the projective orbit will also vanish at a possible zero  which extends to the closure and  gives a contradiction.  Thus,  by Proposition~4 of \cite{Belt}, for any two different numbers $z, w\in \CC$,
\begin{align}
\label{dense}
\overline{\bigg\{\frac{u_n(z) f(\psi^n(z))}{u_n(w) f(\psi^n(w))}\bigg\}}= \CC.
\end{align}
Let $r>0$ be given. Then  $K= \{z\in \CC: |z-z_0|\leq r\}$ is a compact neighbourhood of $z_0$  which also contains $\psi(K)$ since  for each  $z\in K$
\begin{align*}
|\psi(z)-z_0|= |az+b-z_0| \leq |az-az_0|+|az_0-z_0+b| \leq |a|r\leq r.
\end{align*}
 Now, if we set $w= \psi(z), z\in K, z\neq z_0$ and consider the expression in \eqref{dense}
 \begin{align*}
 \bigg|\frac{u_n(z) f(\psi^n(z))}{u_n(w) f(\psi^n(w))}\bigg|=   \bigg|\frac{u(z) f(\psi^n(z))}{u(\psi^n(z)) f(\psi^{n+1}(z))}\bigg| \leq C
 \end{align*}  for all $n\in \NN$ where
 \begin{align*}
 C= \frac{\max_{z\in K}|u(z)|\cdot\max_{z\in K}|f(z)|}{\min_{z\in K}|u(z)|\cdot \min_{z\in K}|f(z)|}.
 \end{align*} This obviously  contradicts the relation in  \eqref{dense}.

It remains to prove the case for $a=1$.  This is rather immediate as  by \eqref{weight}, we have   $u(z)=  u(0) K_{-b}(z) $
and
\begin{align}
\label{normal}
W_{(u,\psi)}f= u(0) K_{-b} f(\psi)= u(0) \|K_{-b} \|_2k_{-b}   f(\psi)
=u(0) e^{\frac{|b|^2}{2}} W_{(k_{-b},\psi)} f.
\end{align}
Then we  show that $W_{(k_{-b},\psi)}$ is an isometry.   Thus, for $f \in \mathcal{F}_p$
\begin{align*}
   \|W_{(k_{-b},\psi)}f\|_p^p=  \frac{p}{2\pi} \|K_{-b}\|_2^{-p} \int_{\CC} |K_{-b}(z)|^p |f(z+b)|^p e^{-\frac{p}{2}|z|^2}dA(z)\quad \quad \quad \quad \\
   =\frac{p}{2\pi} \|K_{-b}\|_2^{-p} \int_{\CC}  |f(z+b)|^p e^{-\frac{p}{2}|z+b|^2} \Big( |K_{-b}(z)|^p e^{\frac{p}{2}|z+b|^2-\frac{p}{2}|z|^2}\Big) dA(z)\quad  \\
   = \|K_{-b}\|_2^{-p} e^{\frac{p|b|^2}{2}}\|f\|_p^p=\|f\|_p^p\quad \quad \quad \quad
   \end{align*} for all $1\leq p<\infty$ .  This shows  that the operator is a linear  isometry, and in fact one can easily show that it is also a  bijective  operator with inverse map  $W_{(k_{-b},\psi)}^{-1}= W_{(k_{b},\psi^{-1})}$.

 It follows from this,  the fact that $u(0)\neq 0$ and \eqref{normal} that  $W_{(u,\psi)}$ is a normal operator.  Then the conclusion  of supecyclicity
   follows directly from the well known Ansari-Bourdon's theorem \cite{AB} and  completes the proof.

 \section{The Ritt's resolvent growth condition for $W_{(u,\psi)}$  on $ \mathcal{F}_p$ }
 The  resolvent  growth condition for an operator  is another  dynamical structure closely related to its iterates, power boundedness and spectral properties. We recall that the resolvent function $R(\lambda, T)$  of an operator $ T$ defined on the resolvent set  $\varrho(T)$ is an operator valued function given by $R(\lambda, T)= (\lambda I- T)^{-1}. $  If  $T$ is power bounded and hence $\sigma(T)$ contained in the closed unit disc, then using the series  expansion of the resolvent function
\begin{align}
\label{Kreiss}
\|R(\lambda, T)\|= \bigg\|\sum_{n=0}^\infty \frac{T^n}{\lambda^{n+1}}\bigg\|\leq \sum_{n=0}^\infty \frac{\|T^n\|}{|\lambda|^{n+1}}= \frac{c}{|\lambda|-1}
\end{align} where $c:= \sup_{n\in\NN_0} \|T^n\|$.
Thus, every  power bounded operator satisfies the so-called Kreiss's resolvent condition. Conversely, for infinite dimensional spaces, condition \eqref{Kreiss} does not imply power boundedness as it only  gives   $\|T^n\|=O(n)$ as  $n\to \infty$; see for example \cite{SH}.   Now we are interested in a stronger resolvent growth condition that conversely  implies  power boundedness,  namely,  the Ritt's resolvent condition \cite{Ritt}. An operator $T$ satisfies such a  condition if there exists  a positive constant $M$ such that
\begin{align}
\label{Ritt}
\|R(\lambda, T)\|\leq \frac{M}{|\lambda-1|}
\end{align} for all $\lambda \in \CC$ with $ |\lambda|>1.$  Both the Kreiss's and Ritt's resolvent conditions   play important roles in numerical analysis; see \cite{BB} and the references therein  for more details.

We now state our next main result.
\begin{theorem}\label{Ritt}  Let $1\leq p<\infty$ and $(u,\psi)$ be a pair  of entire functions on $\CC$ which induces a bounded  weighted composition operator $W_{(u,\psi)}$ on $\mathcal{F}_p$ and hence $\psi(z)= az+b, \ |a|\leq 1$. Then if
\begin{enumerate}
\item $|a|= 1$, then  $W_{(u,\psi)}$ satisfies the Ritt's resolvent condition if and only if any one of the following holds: \begin{itemize}
    \item  $a=1$, $b=0$ and $|u(0)|<1$
    \item $a=1$, $b=0$ and $u(0)=1$
    \item   $a=1$, $b\neq0$ and $ |u(0)|< e^{\frac{-|b|^2}{2}}$
    \item
$a\neq 1$ and $ |u(0)|< e^{\frac{-|b|^2}{2}}$.
\end{itemize}

\item   $W_{(u,\psi)}$ is compact and  satisfies the Ritt's resolvent growth  condition, then   $|u\big(\frac{b}{1-a}\big)|\leq 1$. Conversely, if $|u\big(\frac{b}{1-a}\big)|<1$, then  $W_{(u,\psi)}$ satisfies the Ritt's resolvent growth  condition.
\end{enumerate}
\end{theorem}
\emph{Proof.}
In 1999, Nagy and Zemanek \cite{Nagy} proved  that a bounded  operator $T$ on complex Banach space satisfies the Ritt's resolvent growth condition if and only if it is
power bounded and the difference of  its consecutive iterates satisfy
\begin{align}
\label{Ritt1}
\sup_{n\in \NN}n\|T^{n+1}-T^n\| <\infty.
\end{align}
As before, we  let  $\psi(z)= az+b, \ |a|\leq 1, $ and set $z_0= \frac{b}{1-a}$ when $a\neq1$.
 From  \cite[Theorem~3.1]{TW4}, we recall the following results about the  spectrum of the operators. If
  $W_{(u,\psi)}$  is compact and hence $|a|<1$, then
 \begin{align}
 \label{spectrum}
\sigma(W_{(u,\psi)})=
\Big\{0, \  u\big(z_0\big) a^m, \ \  m\in \NN_0\Big\}.
\end{align}
On the other hand, if $|a|=1$, then
 \begin{align}
\label{spectrum2}
\sigma(W_{(u,\psi)})=
\begin{cases}
\overline{\Big\{ u\big(z_0\big) a^m: m\in \NN_0 \Big\}},  \ a\neq1 \\
  \left\{ z\in \CC: |z|  = |u(0)| e^{\frac{|b|^{2}}{2}}  \right\} , \ \ \ a=1, \ b\neq 0\\
\{u(0)\}, \ \ \ a=1, \ \ b= 0.
\end{cases}
\end{align}
Now from \cite{LY} and \cite[Theorem 4.5.4]{NE},   if $W_{(u,\psi)}$ satisfies the  Ritt's  growth  condition, then its spectrum is  contained in a stolz type domain $\overline{\beta_\theta}$ where $\beta_\theta$ is the interior of convex hull of the set $ \{1 \} $ and the disc $ \{ z\in \D: |z|\leq \sin\theta \}$ and $\theta= \arccos \frac{1} {M}$ where $M$ is the best possible constant in \eqref{Ritt} and $\theta\in [0,\frac{\pi}{2}).$  In particular, we have
\begin{align}
\label{cone} \sigma(W_{(u,\psi)})\cap \mathbb{T} \subset \{1\}
\end{align} where $\mathbb{T}$ denotes the unit circle.

(i) Since Ritt's condition implies power boundedness, by \cite[Theorem~2.1]{TW4},
\begin{align*}
|u(b/(1-a))|= |u(0)|e^{\frac{|b|^2}{2}} \leq 1.
\end{align*}
This together with \eqref{spectrum2} and \eqref{cone} implies  for $|a|=1$ that either $a=1$, $b=0$ and $|u(0)|\leq1$ or $u(0)=1$,  or $a=1$, $b\neq0$ and $ |u(0)|< e^{\frac{-|b|^2}{2}}$ or
$a\neq 1$ and $ |u(0)|< e^{\frac{-|b|^2}{2}}$.

 Conversely, if $a=1$,  $b= 0$ and $|u(0)|\leq 1$, then $W_{(u,\psi)}$ is trivially power bounded. Furthermore, $W_{(u,\psi)}$  reduces to  the multiplication operator $M_u$.   By  \cite[Lemma~2.3]{TM5}, it is known that  $M_u$  is bounded on $\mathcal{F}_p$  if and only if $u$ is a constant function. Thus, if $u= u(0)$, then $\|T^{n+1}-T^n\|= |u(0)|^n |1-u(0)|$ which implies that
   condition \eqref{Ritt1} holds whenever $u(0) =1$ or $|u(0)|<1.$
 On the other hand, if $|u(0)|\leq1$ or $a=1$, $b\neq0$ and $ |u(0)|< e^{\frac{-|b|^2}{2}}$ or
$a\neq 1$ and $ |u(0)|< e^{\frac{-|b|^2}{2}}$, then  applying \eqref{norm} and \eqref{weight}, we have
 \begin{align*}
%\label{ittt}
\| W_{(u,\psi)}^n \|= |u(z_0)|^n= \Big |u(0)e^{\frac{a|b|^2}{a-1}}\Big|^n= \Big(|u(0)|  e^{|b|^2\Re\big(\frac{a }{a-1}\big)}\Big)^n \quad \quad \quad  \\
= |u(0)|^n  e^{n|b|^2\Re\big(\frac{a(\overline{a}-1) }{(a-1)(\overline{a}-1)}\big)}
= |u(0)|^n  e^{\frac{n|b|^2}{2}},
\end{align*}
from which it follows that
\begin{align*}
%\label{pp}
\sup_{n\in \NN}n\|W_{(u,\psi)}^{n+1}-W_{(u,\psi)}^n\|\leq  \sup_{n \in \NN}n \Big(|u(0)| e^{\frac{|b|^2}{2} }\Big)^n+\sup_{n \in \NN}n\Big(|u(0)| e^{\frac{|b|^2}{2} }\Big)^{n+1} <\infty.
\end{align*}
(ii) For $|a|<1$, first observe that for each $n\in \NN$,   the operator $W_{(u,\psi)}^n$ itself is a weighted composition operator  induced by the symbol $( u_n,\psi^n)$ and $W_{(u,\psi)}^n = W_{(u_n,\psi^n)} $. Aiming to check condition \eqref{Ritt1}, we first find an estimate for  the norm of the difference of the two weighted composition operators,
 \begin{align*}
 \big\|W_{(u,\psi)}^{n+1}-W_{(u,\psi)}^n\big\|= \big\| W_{(u_{n+1},\psi^{n+1})}-W_{(u_n,\psi^n)}\big\|.
 \end{align*}
  Applying the difference  to the normalized reproducing kernels $k_w$   and relation \eqref{global}
\begin{align}
\label{diff}
\big\| W_{(u_{n+1},\psi^{n+1})}-W_{(u_n,\psi^n)}\big\|\geq
 \big\|W_{(u_{n+1},\psi^{n+1})}k_w-W_{(u_n,\psi^n)}k_w\big\|_p\nonumber\\
 \geq \Big | u_{n+1}(z) e^{\overline{w}\psi^{n+1}(z)}-u_{n}(z) e^{\overline{w}\psi^{n}(z)}\Big| e^{-\frac{|z|^2+|w|^2}{2}}
\end{align} for all $w, z\in \CC$. In particular, setting  $z= w= z_0$ it readily follows from \eqref{diff} and   \eqref{interplay} that
\begin{align}
\label{diff2}
\big\| W_{(u_{n+1},\psi^{n+1})}-W_{(u_n,\psi^n)}\big\|\geq|u_n(z_0)| |u(z_0)-1| = |u(z_0)|^n |u(z_0)-1|,
\end{align}from which the relation  in \eqref{Ritt1} for  $\| W_{(u,\psi)}\|$ holds only if $|u(z_0)| \leq 1$.

Conversely, by  \cite[Theorem 2.2]{TW4},  $W_{(u,\psi)}$ is power bounded if and only if $|u(z_0)|\leq 1,$ and $\|W_{(u,\psi)}^n\|\simeq |u(z_0)|^n $. Then we repeat the arguments  made for $|a|=1,  a\neq 1$ to show that \eqref{Ritt1} holds whenever $|u(z_0)| <1$, and this completes the proof.

It would be desirable to know whether   the necessity  condition $|u(z_0)|= 1$ could  be sufficient as well. We in fact conjecture that it should be.  In the following we provide several results  in favor of  our conjecture.
\begin{corollary}
\label{cor1}
Let $1\leq p<\infty$ and $(u,\psi)$ be a pair  of entire functions on $\CC$ which induces a bounded  weighted composition operator $W_{(u,\psi)}$ on $\mathcal{F}_p$ and hence $\psi(z)= az+b, \ |a|\leq 1$. If $a=0$, then  $W_{(u,\psi)}$ satisfies the Ritt's resolvent condition on  $\mathcal{F}_p$ if and only if either $u(b)=1$
 or $|u(b)|<1$\end{corollary}
 \emph{Proof}. The necessity of the condition follows from \eqref{diff2} as a particular case. Thus, we  shall verify the sufficiency. A simple computation shows that for each  $f\in \mathcal{F}_p$
 \begin{align*}
 \big\| W_{(u_{n+1},\psi^{n+1})}f-W_{(u_n,\psi^n)}f\big\|_p^p \leq |u(b)|^{pn} \|f\|_p^p |u(b)-1|^p.
  \end{align*}
  It follows that
 \begin{align*}
 \big\| W_{(u_{n+1},\psi^{n+1})}-W_{(u_n,\psi^n)}\big\| \leq |u(b)|^n  |u(b)-1|
  \end{align*}
 from which it is easy to see that \eqref{Ritt1} holds.

Observe that  the multiplication operator $M_u$ is power bounded if and only if $|u(0)|\leq 1$. This  gives the following  consequence.
\begin{corollary}Let $1\leq p<\infty$ and $u$  be an entire functions on $\CC$.  Then a non-trivial   $M_u$   satisfies the Ritt's resolvent condition on $\mathcal{F}_p$ if and only if  it is power bounded.
    \end{corollary}
 The composition operator is one of  the other cases where we  have  $u(z_0)=1$. For this we show that  the  Ritt's resolvent condition  is in fact equivalent to the stronger   unconditional Ritt's condition on $\mathcal{F}_p$.  We recall that an operator $T$ on a Banach space $\mathcal{X}$ satisfies the unconditional Ritt's condition if  there exists a non negative constant K such that
 \begin{align}
 \label{uncond}
 \bigg\|\sum_{n=1} a_n \big(T^n-T^{n-1}\big)\bigg \| \leq K\sup_n\{|a_n|\}
 \end{align} for any finite sequence $(a_n)$ of complex numbers. We note in passing that the notion of the unconditional Ritt's condition is the discrete analogue of the $H^\infty$ calculus for sectorial operators \cite{MN}.   N. Kalton and P. Portal \cite{KK} proved that  the unconditional Ritt's condition  implies the Ritt's resolvent condition in general, but not conversely. But for the composition operators, it turns out that the two are equivalent.
\begin{theorem}
Let $1\leq p<\infty$ and $\psi$ be an  entire function on $\CC$ which induces a bounded   composition operator $C_\psi$ on $\mathcal{F}_p$. Then the following are equivalent.
\begin{enumerate}
\item $C_\psi$ satisfies the Ritt's resolvent condition on $\mathcal{F}_p$;
\item $C_\psi$ is compact or  $C_\psi$ is the identity map on  $\mathcal{F}_p$;
\item $C_\psi$ satisfies the unconditional Ritt's condition on $\mathcal{F}_p$.
\end{enumerate}
 \end{theorem}
 \begin{proof}
  For p=2, the result was proved in \cite{TM8}. We now modify the  proof for all $p's$ and first show that the statements in  (i) and (ii) are equivalent. Recall that $C_\psi$  is bounded if and only if   $\psi(z)= az+b, \ |a|\leq 1$ and $b=0$ whenever $|a|=1$. Compactness is described by the strict inequality $|a|<1$. If $|a|=1$, then by Theorem~\ref{Ritt} the composition operator $ C_\psi$  satisfies the Ritt's resolvent condition on $\mathcal{F}_p$ if and only if $a=1,$ which means $\psi$ reduces to the identity operator. Thus,  we shall  proceed to show the equivalency for the case when $|a|<1$. The statement (i) implies (ii) follows again from part (ii) of Theorem~\ref{Ritt}.   We need to prove the converse.
 For each $z\in \CC$ and $n\in \NN$, let $\delta_n$ be a positive number  such that $a^{n+1}z+\frac{b(1-a^{n+1})}{1-a}\in D\big(a^nz+ \frac{b(1-a^n)}{1-a}, \delta_n\big)$.  An explicit expression for $\delta_n$ will be given latter. Then for any $f\in \mathcal{F}_p$, by the Mean Value Theorem
\begin{align}
\label{allp}
\big| C_{\psi^{n+1}}f(z)-C_{\psi^n}f(z)\big|^p = \bigg| f\bigg(a^{n+1}z+\frac{b(1-a^{n+1})}{1-a}\bigg)- f\bigg( a^nz+ \frac{b(1-a^n)}{1-a} \bigg)\bigg|^p \nonumber\\
\leq
 \Big|a^{n+1}z+\frac{b(1-a^{n+1})}{1-a}- a^nz- \frac{b(1-a^n)}{1-a}\Big|^p \sup_{w\in D\big(a^nz+ \frac{b(1-a^n)}{1-a}, \beta_n\big)} |f'(w)|^p\quad \quad \nonumber\\
= |a^n z (a-1)+ba^n|^{p} \sup_{w\in D\big(a^nz+ \frac{b(1-a^n)}{1-a}, \beta_n\big) } |f'(w)|^p \quad \quad \quad \quad
\end{align} for some positive  $\beta_n<\delta_n$ which will also  be  fixed latter.\\
Using the reproducing property,
\begin{align*}
f(w)=\langle f, K_w\rangle= \frac{1}{\pi} \int_{\CC} f(z) e^{w\overline{z}}e^{-|z|^2} dA(z)
\end{align*} and
\begin{align*}
|f'(w)|\leq \frac{1}{\pi} \int_{\CC} |z||f(z)| e^{\Re(w\overline{z})}e^{-|z|^2} dA(z)\quad \quad \quad \quad \quad \quad \quad \quad \quad \\
\leq \frac{1}{\pi}\bigg(\sup_{z\in \CC} |f(z)| e^{-\frac{|z|^2}{2}}\bigg) \int_{\CC} |z| e^{\Re(w\overline{z})}e^{-\frac{|z|^2}{2}} dA(z)\\
\leq\frac{1}{\pi}\|f\|_p \int_{\CC}|z| e^{\Re(w\overline{z})}e^{-\frac{|z|^2}{2}} dA(z)
\leq\frac{1}{\pi}e^{\frac{|w|^2}{2}}\|f\|_p \int_{\CC} (1+|z|) e^{-\frac{|z-w|^2}{2}} dA(z).
\end{align*}
On the other hand,
\begin{align*}
1+|z|\leq  1+|w-z|+|w|\leq \big(1+|w-z|\big) (1+|w|)
\end{align*} and hence
\begin{align*}
\frac{1}{\pi} \int_{\CC} (1+|z|) e^{-\frac{|z-w|^2}{2}} dA(z)\leq (1+|w|) \frac{1}{\pi} \int_{\CC}(1+|w-z|) e^{-\frac{|z-w|^2}{2}} dA(z  \\
= (1+|w|)\left(1+ 2^{\frac{3}{2}}\Gamma\Big(1+\frac{1}{2}\Big)\right)=   (\sqrt{2\pi}+1)(1+|w|) \quad \quad
\end{align*} where $\Gamma$ refers to the gamma function. Thus, we deduce
\begin{align*}
%\label{der}
|f'(w)| \leq  (\sqrt{2\pi}+1)(1+|w|)e^{\frac{|w|^2}{2}} \|f\|_{p}.
\end{align*}
Now, since $w\in D\big(a^nz+ \frac{b(1-a^n)}{1-a}, \beta_n\big)$ we have
\begin{align*}
|w|\leq \bigg|w-a^nz+ \frac{b(1-a^n)}{1-a}\bigg|+\bigg|a^nz+ \frac{b(1-a^n)}{1-a}\bigg|\quad \quad \quad \quad \quad \quad \quad \quad \quad \quad \quad  \\
\leq \beta_n+ \Big|a^nz+ \frac{b(1-a^n)}{1-a}\Big|\leq \beta_n+|a|^n|z|+\frac{2|b|}{|1-a|}=:\alpha_n(z)\quad \quad \quad \quad \quad
\end{align*}
Now taking all these estimates in \eqref{allp}
\begin{align*}
\frac{1}{\pi}\int_{\CC}\big| C_{\psi^{n+1}}f(z)-C_{\psi^n}f(z)\big|^p e^{-\frac{p}{2}|z|^2}dA(z) \leq \quad \quad \quad \quad \quad \quad \quad \quad\quad \quad \quad \quad\quad \quad \quad \quad\nonumber \\
2|a|^{pn} \|f\|_{p}^{p} (\sqrt{2\pi}+1)^{p} \int_{\CC} |(a-1)z+b|^p |1+\alpha_n(z)|^p  e^{\frac{p}{2}(\alpha_n^2(z)-|z|^2)} dA(z) \\
\leq 2|a|^{pn} \|f\|_{p}^{p} (\sqrt{2\pi}+1)^{p} I_n \quad \quad \quad \quad \quad \quad \quad \quad \quad \quad
\end{align*} where
\begin{align}
\label{conv}
I_n=\int_{\CC}\Big( |z|^2+|z|^2\alpha_n^2(z)+|b|^2+ |b|^2\alpha_n^2(z) \Big) e^{\frac{p}{2}(\alpha_n^2(z)-|z|^2)} dA(z).
 \end{align}
 Next,  we show that the integral above is uniformly bounded by a positive number independent of  $n\in \NN$.  To this end, observe that   since  $a^{n+1}z+\frac{b(1-a^{n+1})}{1-a}\in D\big(a^nz+ \frac{b(1-a^n)}{1-a}, \delta_n\big)$
 \begin{align*}
 \bigg|a^{n+1}z+\frac{b(1-a^{n+1})}{1-a}- a^nz- \frac{b(1-a^n)}{1-a}\bigg|= |a^n z (a-1)+ba^n|\quad \quad \quad \\
 \leq |a|^n |a-1||z|+|a|^n|b|
 \end{align*}   $\delta_n$ can be  taken to be $|a|^n |a-1||z|+|a|^n|b|$. Thus,  $\beta_n$ can also be chosen in such a way that $\beta_n=\alpha_n |z|+|a|^n|b|$ for some  $0<\alpha_n< |a|^n |a-1|$  whenever  $z\neq 0$.
 Then
 \begin{align}
 \label{esti}
 \alpha_n^2(z)
 < \bigg( |a|^n( |a-1|+1)|z| +|a|^n|b| +\frac{2|b|}{|1-a|}\bigg)^2 = |a|^{2n}( |a-1|+1)^2 |z|^2 \nonumber\\
 + 2 |a|^{n}|z|( |a-1|+1)\bigg(|a|^n|b| +\frac{2|b|}{|1-a|}\bigg)
 +  \bigg(|a|^n|b| +\frac{2|b|}{|1-a|}\bigg)^2.
 \end{align}
Given the estimate in \eqref{esti} and  the exponential integrating weight in \eqref{conv}, we claim that
 \begin{align}
 \label{negative}
 |a|^{2n}\big( |a-1|+1\big)^2-1 \leq 0
 \end{align} for all $n$.  Setting $a= |a|e^{i\theta}$  we  write
\begin{align*}|a|^{2n}\big(( |a-1|+1\big)^2-1= |a|^{2n}\Big( |a|\sqrt{(\cos\theta-1)^2+(\sin\theta)^2}+1\Big)^2-1.
\end{align*}
   Thus, it is enough to show that $|a|^{2}(2|a|+1)^2-1 \leq 0$. But this obviously holds  since  $r=|a|<1$ and the values  of the function $g(r)= r^2(2r+1)^2-1$ lie in the interval $[-1,0]$.
  Now, integrating with polar coordinates shows that the integral in  \eqref{conv}   is  uniformly bounded by positive number  $C$ independent of $n$. Therefore,
  \begin{align}
   \label{expon}
  \big\| C_{\psi^{n+1}}-C_{\psi^n}\big\|^{p} \leq C  |a|^{pn}.
  \end{align} This together with $|a|<1$, the relation in \eqref{Ritt1} holds for $C_{\psi}$.

Next we show that the statements in (i) and (iii) are also equivalent.  From the discussion above, we have already mentioned that (iii) implies (i). Thus, it remains to show (i) implies (iii).  But this is rather immediate since by \eqref{expon},  we observe that $\|C_{\psi^n}-C_{\psi^{n+1}}\|$ is bounded by an exponentially decreasing sequence with $n$. Thus
\begin{align*}
\bigg\|\sum_{n=1} a_n \big(C_{\psi^n}- C_{\psi^{n-1}}\big)\bigg \| \leq \sup_n\{|a_n|\}\sum_{n=1} \|C_{\psi^n}-C_{\psi^{n-1}}\|\lesssim  \sup_n\{|a_n|\}.
\end{align*}
\end{proof}


\begin{thebibliography}{BRSHZE}
\bibitem{AB}  I. Ansari, P. S. Bourdon, Some properties of cyclic operators, Acta Sci. Math. (Szeged), 63(1997), 12, 195--207.

\bibitem{Belt} M. J. Beltr\'an-Meneu,   E. Jord\'a,   M. Murillo-Arcila, Supercyclicity of weighted composition operators on spaces of continuous functions,  Collect. Math.  71(2020), 493--509.

\bibitem{BB} N. Borovykh, D. Drissi,  M. N. Spijker, A note about Ritt’s condition, related resolvent conditions
and power bounded operators. Num. Funct. Anal. Optim., 21(2000), 425--438 .


    \bibitem{KG} K. Guo, K. Izuchi, Composition operators on Fock type spaces, Acta Sci. Math. (Szeged).,  74(2008),  807--828.


\bibitem{Tle} T. Le, Normal and isometric weighted composition operators on the Fock space, Bull. London. Math. Soc.,  46(2014), 847--856.

\bibitem{LY} Yu. Lyubich. Spectral localization, power boundedness and invariant
subspaces under Ritt's type condition. Studia Mathematica, 143(2)(1999) 153--167.


\bibitem{KK} N. J. Kalton,  P. Portal, Remarks on $\ell_1$ and $\ell_\infty$-maximal regularity for power bounded operators, J.
Aust. Math. Soc.,  84 (2008), 345--365.

\bibitem{MN} A. McIntosh,  Operators which have an $H_{\infty}$ functional calculus. Miniconference on Operator Theory and Partial Differential Equations, 210--231, Centre for Mathematics and its Applications, Mathematical Sciences Institute, The Australian National University, Canberra AUS, 1986.


    \bibitem{TM4} T. Mengestie, Carleson type measures for Fock--Sobolev spaces,
Complex Anal. Oper. Theory.,  8(2014), 6, 1225--1256.
\bibitem{TM8} T.  Mengestie, Resolvent growth condition for composition operators on the Fock space, Annals of Functional Analysis, 11(2020),  947--955.
\bibitem{TM5} T.  Mengestie, Spectral properties  of Volterra-type integral operators on Fock--Sobolev  spaces,  Journal of the Korean Mathematical Society, 54(2017), no. 6,  1801--1816.


 \bibitem{TW2} T. Mengestie,  W. Seyoum, Spectral Properties of Composition Operators on Fock-Type Spaces, Quaestiones Mathematicae, DOI: 10.2989/16073606.2019.1692092.

 \bibitem{TW1} T.  Mengestie, W. Seyoum, Topological and dynamical properties of composition operators, Complex Anal. Oper. Theory (2020) 14: 2. DOI:10.1007/s11785-019-00961-8

\bibitem{TMMW}  T. Mengestie,  M. Worku, Topological structures of generalized Volterra-type integral operators,  Mediterr. J. Math.,  (2018) 15: 42. https://doi.org/10.1007/s00009-018-1080-5

\bibitem{Nagy} B. Nagy,  J. A. Zemanek, A resolvent condition implying power boundedness, Studia
Math.,  134(1999), 143--151.
\bibitem{NE} O. Nevanlinna, Convergence of iterations for linear equations, Lecture Notes in Mathematics,
ETH Z\"{u}rich, Birkh\"{a}user, Basel, 1993.

\bibitem{SH} A. L. Shields, On M\"{o}bius bounded operators, Acta Sci. Math. (Szeged), 40(1978), 371--374.

 \bibitem{TW4}W. Seyoum, T. Mengestie,   Spectral and uniform mean  ergodicity of weighted composition operators  on Fock-Type Spaces, Preprint 2020.

\bibitem{Ritt} R. K. Ritt, A condition that $\lim_{n\to \infty} n^{-1}T^nn =0$, Proc. Amer. Math. Soc.,  4(1953),  898--899.

\bibitem{UK} S. Ueki,  Weighted composition operators on the Fock space, Proceedings of the American Mathematical Society.,  135(2007), 1405--1410.

\end{thebibliography}
\end{document}